\documentclass[12pt]{amsart}
			\usepackage{fixmath}

			\usepackage[T1]{fontenc}
			\usepackage{ucs}
			\usepackage[utf8x]{inputenc}
			\usepackage[english]{babel}
			
			\usepackage{microtype}
			\usepackage{babel}
			\usepackage{textgreek}
			
			\usepackage{amssymb}
			
			\usepackage{mathtools} 
			\usepackage{dsfont} 
			
			\usepackage{enumerate}

			\usepackage{geometry} 

			\usepackage{amsfonts, amsmath, amsthm, amssymb}
			\usepackage{mathrsfs}
			\usepackage[colorlinks,linkcolor=blue]{hyperref} 
			\usepackage[msc-links]{amsrefs} 
			\usepackage{doi} 
			\renewcommand{\PrintDOI}[1]{DOI \doi{#1}}
			
\makeatletter
\newtheorem*{rep@theorem}{\rep@title}
\newcommand{\newreptheorem}[2]{%
\newenvironment{rep#1}[1]{%
 \def\rep@title{#2 \ref{##1}}%
 \begin{rep@theorem}}%
 {\end{rep@theorem}}}
\makeatother
			\newtheorem{theorem}{Theorem}
			\newreptheorem{reptheorem}{Theorem}
			
			\newtheorem{lemma}{Lemma}
			
			\newtheorem{proposition}{Proposition}
			\newtheorem{corollary}{Corollary}
			
			\theoremstyle{definition}

			\theoremstyle{remark}
			\newtheorem{remark}{Remark}

			\DeclareMathOperator{\Pn}{Poisson}

			\DeclareMathOperator{\Var}{Var}
			\DeclareMathOperator{\Cov}{Cov}

			\providecommand{\Prob}[1]{\mathbb{P}\{#1\}}

			\providecommand{\card}[1]{\texttt{\#}#1}
			
			\providecommand{\norm}[1]{\lVert#1\rVert}

			\newcommand{\Ex}{\mathbb{E}}
			\providecommand{\Exc}[2]{\mathbb E\left[#1\middle| #2\right]}

			\newcommand{\Ewens}[1]{\text{Ewens}(#1)}
			
			\providecommand{\abs}[1]{\lvert#1\rvert}
			
			\providecommand{\risfac}[2]{#1^{\overline{#2}}} 
		
\begin{document}
	\title[Number of cycles in random permutations]{The number of cycles in a random permutation and the number of segregating sites\\ jointly converge to the Brownian sheet}

	\author{Helmut H.~Pitters}
	\address{Mathematical Institute, University of Mannheim}
	\email{helmut.pitters@mail.uni-mannheim.de}
	\thanks{The author acknowledges financial support by the DFG RTG 1953.}
			
	
			\begin{abstract}
  Consider a random permutation of $\{1, \ldots, \lfloor n^{t_2}\rfloor\}$ drawn according to the Ewens measure with parameter $t_1$ and let $K(n, t)$ denote the number of its cycles, where $t\equiv (t_1, t_2)\in\mathbb [0, 1]^2$.
  
  Next, consider a sample drawn from a large, neutral population of haploid individuals subject to mutation under the infinitely many sites model of Kimura whose genealogy is governed by Kingman's coalescent. Let $S(n, t)$ count the number of segregating sites in a sample of size $\lfloor n^{t_2}\rfloor$ when mutations arrive at rate $t_1/2$.
  
  Our main result comprises two different couplings of the above models for all parameters $n\geq 2,$ $t\in [0, 1]^2$ such that in both couplings one has weak convergence of processes as $n\to\infty$
\begin{align*}
  \left\{\frac{(K(n, s), S(n, t))-(s_1s_2, t_1t_2)\log n}{\sqrt{\log n}}, s, t\in [0, 1]^2\right\}\to\{(\mathscr B(s), \mathscr B(t)), s, t\in [0, 1]^2\},
\end{align*} where $\mathscr B$ is a one-dimensional Brownian sheet. This generalises and unifies a number of well-known results.

\end{abstract}

		\maketitle

	
		
			
\section{Introduction}
Fix a natural number $n\in\mathbb N\coloneqq \{1, 2, \ldots \}$ and a real number $t_1>0$. A random $(n, t_1)$ Ewens permutation is an element chosen at random in the symmetric group $\mathfrak S_n$ on the numbers $1, \ldots, n$ such that $\sigma\in\mathfrak S_n$ is picked with probability
\begin{align*}
  \frac{t_1^{\card\sigma}}{\risfac{t_1}{n}}.
\end{align*}
Here $\card\sigma$ denotes the number of cycles in $\sigma$, and for any real $x\in\mathbb R$  and $n\in\mathbb N$ denote by $\risfac{x}{n}\coloneqq x(x+1)\cdots (x+n-1)$ the $n$th rising factorial power of $x$. We study the asymptotic  behaviour for large $n$ of the random number $K(n, t)$ of cycles in a $(\lfloor n^{t_2}\rfloor, t_1)$ Ewens permutation, where $t\equiv (t_1, t_2)\in\mathbb R_+^2$.

Moreover, we consider the so-called Kingman $n$-coalescent or $n$-coalescent for short. This process first appeared in the context of population genetics where it is used to model the genealogy of a sample of size $n$ drawn from a large neutral population of haploid individuals. There are two well-known couplings of the sequence of $n$-coalescent processes for all $n\geq 2$, the natural coupling and the temporal coupling. Unless otherwise specified, the quantities that we study and our results hold true both in the natural coupling and in the temporal coupling. The individuals are subject to mutations under the infinitely many sites model of Kimura. Here we focus on the random nuber $S(n, t)$ of mutations affecting a sample of size $\lfloor n^{t_2}\rfloor$ when mutations arrive at rate $t_1/2>0$.

Our first main result provides a coupling of $(K(n, t), S(n, t))$ simultaneously for all parameters $n\geq 2$ and $t$ as follows. Let $\mathfrak P$ denote a Poisson point process on $\mathbb R_+^2$ with intensity $\lambda\otimes \frac{1}{2}\lambda$, where $\lambda$ denotes Lebesgue measure on $\mathbb R.$ Moreover, let $(L_n, n\geq 2)$ denote the total branch lengths of the $n$-coalescent tree with the agreement that $L_1\coloneqq 0$, such that $(L_n, n\geq 2)$ is independent of $\mathfrak P$.

For any $t\in \mathbb R_+^2$ and $k\geq 2$ set
\begin{align}
M_k(t_1) \coloneqq \card(\mathfrak P\cap [L_{k-1}, L_k)\times [0, t_1)),
\intertext{and}
B_k(t_1)\coloneqq \mathbf 1\{M_k(t_1)\geq 1\}.
\end{align}

\begin{theorem}\label{thm:coupling}
	For any $n\geq 2$ and $t\in\mathbb R_+^2$ we have
	\begin{align}\label{eq:definitions_k_s}
	K(n, t) =_d 1+\sum_{k=2}^{\lfloor n^{t_2}\rfloor}B_k(t_1), \qquad S(n, t)=_d\sum_{k=2}^{\lfloor n^{t_2}\rfloor} M_k(t_1).
	\end{align}
\end{theorem}
\begin{remark}
	Theorem~\ref{thm:coupling} in fact entails two couplings of $K(n, t)$ and $S(n, t)$, depending on whether the sequence $(L_n, n\geq 2)$ stems from the natural coupling or the temporal coupling of the $n$-coalescent processes. In the temporal coupling the summands $B_k$ and $M_k$ in equation~\eqref{eq:definitions_k_s} are independent.
\end{remark}
\begin{remark}
	From the point of view of population genetics the quantity $K(n, t)$ may be interpreted as the number of alleles in the infinitely many alleles model stemming from mutations that are spread at rate $t_1/2$ in the $\lfloor n^{t_2} \rfloor$-coalescent above.
\end{remark}

Let $\{\mathscr B(s), s\in[0, 1]^2\}$ denote a standard Brownian sheet on the unit square. Our second main result shows weak convergence of the joint processes $(K(n, s), S(n, t))$ after suitable rescaling as $n\to\infty$.
\begin{theorem}\label{thm:main_result}
As $n\to\infty$ we have weak convergence of processes
	\begin{align}
	\left\{\frac{(K(n, s), S(n, t))-(s_1s_2, t_1t_2)\log n}{\sqrt{\log n}}, s, t\in [0, 1]^2\right\}\to\{(\mathscr B(s), \mathscr B(t)), s, t\in [0, 1]^2\},
	\end{align}
on the space $D_4$ in the S-topology, where $\mathscr B$ is a one-dimensional Brownian sheet on the unit square.
\end{theorem}
Our proof relies on the asymptotic behaviour of the total branch lengths of the $n$-coalescent in the temporal coupling, $L_n^\dagger$, and in the natural coupling, $L_n^\star$. It is well-known that in the temporal coupling convergence holds almost surely and in $L^p$ for any $p>0$,
\begin{align*}
  L_n^\dagger-2\log n\to 2\mathscr L^\dagger
\end{align*}
as $n\to\infty,$ where $\mathscr L^\dagger$ obeys a standard Gumbel law, that is the law with distribution function $\exp(-\exp(-x)),$ $x\in\mathbb R,$ on the real line, cf.~\cite{MoehlePitters2015}. In the natural coupling, Pfaffelhuber, Wakolbinger and Weisshaupt in~\cite{PfaffelhuberWakolbingerWeisshaupt2011} showed convergence
\begin{align*}
L_n^\star-2\log n\to 2\mathscr L^\star
\end{align*}
in $L^2$ as $n\to\infty$. The limit $\mathscr L^\star$ also obeys a standard Gumbel law. We complement these results as follows.
\begin{theorem}\label{thm:branch_lengths_convergence}
	The convergence
	\begin{align}
	L_n^\star -2\log n\to 2\mathscr L^\star
	\end{align}
	of the total branch lengths in the natural coupling holds almost surely. Moreover, the limits $\mathscr L^\dagger$ and $\mathscr L^\star$ are equal almost surely.
\end{theorem}

This work is layed out as follows. In section~\ref{subsec:cycles} we give a formal definition of a Ewens permutation. We review a collection of well-known results on the asymptotic behaviour of $K(n, t)$ and present our results on the functional CLT for $K(n, t)$. In section~\ref{subsec:sites} we define the number of segregating sites $S(n, t).$ We review its law of large numbers and the central limit theorem obtained by Watterson and present our functional CLT for $S(n, t)$. In section~\ref{sec:coupling} we provide the coupling of $K(n, t)$ and $S(n, t)$ and show their joint convergence. The proofs of our results are found in section~\ref{sec:proofs}.

\section{Results}
We review some classical limit theorems on the number of cycles in a Ewens permutation and the number of segregating sites. Next we present our results that generalise these theorems. We write $X^\dagger$ for a quantity considered in the temporal coupling, and $X^\star$ for the corresponding quantity considered in the natural coupling. We write $X$ if the statement about $X$ holds true both in the natural coupling and in the temporal coupling.

\subsection{Number of cycles in a random permutation}\label{subsec:cycles}
Random permutations and their cycle structure have been studied extensively and have a long history. One of the most celebrated families of probability measures on the symmetric group $\mathfrak S_n$ is the so-called Ewens measure parameterised by some parameter $t_1> 0.$ We say that $\Sigma(n)\equiv\Sigma(n, t_1)$ is governed by the $(n, t_1)$ Ewens distribution on $\mathfrak S_n$ if for any $\sigma\in\mathfrak S_n$
\begin{align}\label{eq:def_ewens}
  \Prob{\Sigma(n)=\sigma} &= \frac{t_1^{\card\sigma}}{\risfac{t_1}{n}}\mathds 1\{\sigma\in\mathfrak S_n\}.
\end{align}
In this case we write $\Sigma(n)\sim\Ewens{n, t_1}$. Here for any event $A$, $\mathds 1A$ denotes its indicator which equals one if $A$ occurs and zero otherwise. We agree on setting $\Sigma(n, 0)$ to be the identity map on $\{1, \ldots, n\}$.

\begin{remark}
	Our notation differs from the notation in the literature where the parameter $t_1$ is usually denoted by $\theta.$
\end{remark}

The so-called Feller's coupling, cf.~\cite{ArratiaBarbourTavare2003}, provides a concrete construction of a sequence $\Sigma^F(1), \Sigma^F(2), \ldots$ of random permutations on a common probability space, such that for any $n$ the random permutation $\Sigma^F(n)$ obeys a $(n, t_1)$ Ewens distribution. The superscript $F$ indicates the Feller's coupling of the sequence of random permutations. To this end, consider the standard cycle notation of a permutation. Then $\Sigma^F(1)$ is nothing but the cycle $(1)$. The label $2$ is either inserted in this cycle to the right of $1$ with probability $1/(t_1+1),$ or forms a new cycle with probability $t_1/(t_1+1),$ and one obtains $\Sigma^F(2)$. If $\Sigma^F(k)$ is constructed for some $k\in\mathbb N,$ the label $k+1$ is placed to the right of a randomly chosen integer in $1, \ldots, k$ in the cycle containing this integer, or forms a new cycle with probability $t_1/(t_1+k-1)$. Setting
\[ 
\xi_k\coloneqq \mathds 1\{\text{ label $k$ starts a new cycle } \},\qquad k\geq 2,
\]
$\xi_2, \xi_3, \ldots$ is a sequence of independent Bernoulli r.v.s with $\xi_k$ having success probability $t_1/(t_1+k-1)$. Then $K^F(n, t_1)\coloneqq\card\Sigma(n)$ is nothing but the sum $1+\xi_2+\cdots+\xi_n$.
It turns out that the $K^\dagger(n, t),$ $n\geq 2,$ that we derive from the temporal coupling of the $n$-coalescents have the same dependency structure as the $K^F(n, t)$ derived from Feller's coupling.

In what follows we focus on $K(n)\coloneqq \card\Sigma(n)$, the number of cycles in the random permutation $\Sigma(n)$. Feller's coupling shows
\begin{align}
  \Ex K_n &= \sum_{k=1}^n \frac{t_1}{t_1+k-1},\qquad \Var(K_n) = \sum_{k=2}^n \frac{t_1(k-1)}{(t_1+k-1)^2}.
\end{align}

More particularly, we are interested in the asymptotic behaviour of $K(n)$ for large $n$. However, for properties of $K(n)$ for fixed $n$ the reader is referred to~\cite[Section 5.2]{ArratiaBarbourTavare2003}. Both $\Sigma(n)$ and $K(n)$ depend on the parameter $t_1$, but we drop this dependency in our notation until later.


Korwar and Hollander~\cite{KorwarHollander1973} proved a strong law of large numbers
\begin{align}
  \frac{K(n)}{\log n}\to_{a.s.} t_1,
\end{align}
as $n\to\infty,$ where $\to_{a.s.}$ denotes almost sure convergence. In their work the $K(n)$ are distinct observations in a sample of size $n$ from a Dirichlet process and thus coupled in $n$. For $t_1=1$ Goncharov~\cite{Goncharov1944} and Shepp and Lloyd~\cite{SheppLloyd1966} proved, using generating functions, the central limit theorem (CLT)
\begin{align}
  \frac{K(n)-\log n}{\sqrt{\log n}}\to_d N,
\end{align}
as $n\to\infty,$  where $\to_d$ denotes convergence in distribution and $N$ is a standard Gaussian random variable. This CLT was also shown by Feller~\cite{Feller1945} and R\'{e}nyi~\cite{Renyi1962} via the Lindeberg-Feller CLT, and by Kolchin~\cite{Kolchin1971, Kolchin1986} using a representation with random allocations of particles into cells.
To proceed we need more sophisticated notation, namely for $t\equiv (t_1, t_2)\in [0, 1]^2$ let $K(n, t)$ denote the number of cycles in $\Sigma(\lfloor n^{t_2}\rfloor, t_1)$.

For Feller's coupling, Hansen~\cite{Hansen1990} proved a functional CLT, namely weak convergence on the Skorokhod space $D_1$ to the Wiener measure as $n\to\infty$ of the process
\begin{align}
  \left\{\frac{K^F(n, t)-t_1t_2\log n}{\sqrt{t_1\log n}}, t_2\in [0, 1]\right\}.
\end{align}
Alternative proofs of this result were given by Donnelly, Kurtz and Tavar\'{e}~\cite{DonnellyKurtzTavare1991}, and Arratia and Tavar\'{e}~\cite{ArratiaTavare1992}.
For the special case $t_1=1$ this functional central limit theorem was obtained earlier by DeLaurentis and Pittel~\cite{DeLaurentisPittel1985}.

We can now state our first result, Theorem~\ref{thm:functionalclt_number_of_cycles}, which is a functional CLT for $K(n)$ in both the natural coupling and the temporal coupling of $K(n, t)$ in all parameters $n$ and $t$ that we provide in section~\ref{sec:coupling}. Theorem~\ref{thm:functionalclt_number_of_cycles} generalises and unifies the asymptotic results on $K(n)$ in the literature that we have cited above. It follows directly from our main result, Theorem~\ref{thm:main_result}.

\begin{theorem}[Functional CLT for the number of cycles]\label{thm:functionalclt_number_of_cycles}
  As $n\to\infty$ we have weak convergence of processes
  \begin{align}\label{eq:brownian_sheet_fdd}
    \left\{\frac{K(n, t)-t_1t_2\log n}{\sqrt{\log n}}, t\in\mathbb [0, 1]^2 \right\} \to \{\mathscr B(t), t \in\mathbb [0, 1]^2\}
  \end{align}
  on the space $D_2$ in the S-topology, where $\mathscr B$ is a one-dimensional Brownian sheet on the unit square.
\end{theorem}
\begin{remark}
	Theorem~\ref{thm:functionalclt_number_of_cycles} holds true both in the natural coupling and in the temporal coupling. To the best of our knowledge, the previous results on the asymptotic behaviour of $K(n, t)$ that we reviewed earlier all correspond to the temporal coupling in our setting.
\end{remark}
The definition of the space $D_q$, $q\in\mathbb N,$ and the S-topology is reviewed in section~\ref{sec:multiparameter_processes}.
We now turn to the second model, the number of segregating sites.

\subsection{Number of segregating sites}\label{subsec:sites}
In large neutral populations of haploid individuals the genealogy of a sample of $n$ individuals is often modeled by the $n$-coalescent $\Pi_n$. A verbal description of this stochastic process is as follows. Picture the individuals in the sample labeled $1, \ldots, n$ with a line of descent emanating from each individual and growing at unit speed. At rate one any pair of individuals merges, i.e.~their lines of descent merge into a single line representing the most recent common ancestor of this pair. After the first merger the process continues with $n-1$ lines of descent following the same dynamics as before. It is clear from this description that the genealogy of a sample of $n$ individuals may be represented as a (random) rooted tree with $n$ leafs labeled $1, \ldots, n.$ Let us now turn to a more formal definition of the $n$-coalescent.

Kingman's coalescent $\Pi=\{\Pi(t), t\geq 0\}$ is a Markov process with state space the set partitions of $\mathbb N$ defined as follows. For any $n\geq 2$ the restriction $\Pi_n^\star=\{\Pi_n^\star(t), t\geq 0\}$ of $\Pi$ to $[n]$ is a Markov process on the partitions of $[n]$ such that a transition from $\pi$ to $\pi'$ occurs at unit rate if $\pi'$ can be obtained by merging two blocks in $\pi$, and no other transitions occur. This is referred to as the natural coupling of the $n$-coalescents $\Pi_n^\star,$ $n\geq 2$. Let $(L_n^\star, n\geq 2)$ denote the sequence of total branch lengths in the natural coupling, where $L_n^\star$ is the length of all branches of the tree corresponding to $\Pi_n^\star$. We now turn to the temporal coupling. Let
\[T_k \coloneqq \inf\{t\geq 0 \colon \card\Pi(t)=k   \}\qquad k\geq 1,\]
be the first time that $\Pi$ reaches a state of $k$ blocks. Let
\begin{align}
\tau_k\coloneqq T_{k-1}-T_{k}\qquad k\geq 2,
\end{align}
denote the time spent by $\Pi$ in a state of $k$ blocks. The inter-arrival times $\tau_2,  \tau_3, \ldots$ form a sequence of independent exponentials such that $\tau_k$ has parameter $\binom{k}{2}.$ The total length of branches of the $n$-coalescent tree in the natural coupling is given by
\begin{align}\label{eq:branch_lengths_natural_coupling}
L_n^\star \coloneqq \sum_{k=2}^n \mathbb B_k(n)\tau_k,
\end{align}
where $\mathbb B_k(n)$ counts the number of branch segments that support leaves with labels in $1, \ldots, n$ among the branch segments present when the coalescent has $k$ branches in total.

Let $\Pi_n^\dagger$ denote the restriction of $\{\Pi(T_n+t), t\geq 0\}$ to $[n]$. It turns out that $\Pi_n^\dagger$ is again a Markov chain. From their definition it is immediate that $\Pi_n^\star$ and $\Pi_n^\dagger$ are equal in distribution. The $(\Pi_n^\dagger, n\geq 2)$ are the $n$-coalescents in the temporal coupling. Moreover, the total length of branches of the $n$-coalescent tree in the temporal coupling is given by
\begin{align}\label{eq:branch_lengths_temporal_coupling}
L_n^\dagger \coloneqq \sum_{k=2}^n k\tau_k.
\end{align}
		
			In addition to the genealogy mutations are modeled as follows. Conditionally given the genealogical tree (or coalescent tree), throw down points onto the branches of the tree (identified with intervals of the real line) according to a Poisson point process that is independent of Kingman's coalescent $\Pi$ and has constant intensity $t_1/2>0$, the so-called mutation rate. Each point of the Poisson process is then interpreted as a mutation affecting any leaf (the individual in the sample) with the property that the unique path connecting the leaf to the root of the tree crosses that mutation. A formal way to define this procedure is to identify Kingman's coalescent with a random ultrametric space on which a Poisson process can then be defined. However, this is beyond the scope of this article, and we refer the interested reader to Evans' lecture notes~\cite{Evans2007} instead.
			
			We restrict ourselves to the infinitely many sites model of Kimura~\cite{Kimura1969}. According to Kimura's model each mutation is thought of as acting on one of infinitely many sites, i.e.~each jump of the Poisson process on the tree introduces a mutation on a site where no mutation was previously observed. For detailed expositions of probabilistic models for the evolution of DNA sequences the interested reader is referred to Durrett~\cite{Durrett2008}, Etheridge~\cite{Etheridge2011}, and Tavar\'{e}~\cite{Tavare2004}.
			
			We can write the number of segregating sites in the $\lfloor n^{t_2}\rfloor$-coalescent as
			\begin{align}\label{eq:sites_sum}
				S(n, t) &= \sum_{k=2}^{\lfloor n^{t_2}\rfloor} M_{k}(t_1).
			\end{align}
			In the temporal coupling, $M^\dagger_{k}(t_1)$ counts the number of segregating sites appearing while $\Pi$ has $k$ blocks. In the natural coupling, $M^\star_{k}(t_1)$ counts the number of segregating sites appearing while $\Pi_n^\star$ has $k$ blocks.
			
			Properties of $M_k(t_1)$ can be readily derived in the temporal coupling. In terms of the coalescent tree, $M_{k}^\dagger(t_1)$ is the number of mutations falling on the $k$ segments of branches of length $\tau_{k}$ each, hence the $M_{2}^\dagger(t_1), M_{3}^\dagger(t_1), \ldots$ are independent. If $\Pi$ is in a state of $k$ blocks, the probability to see a mutation before a merger is $t_1/(t_1+k-1)$ since a mutation occurs at rate $kt_1/2$, whereas a merger happens at rate $\binom{k}{2}.$ Consequently, $M_k(t_1)$ counts the number of failures until the first success in a sequence of independent Bernoulli trials with success probability $(k-1)/(t_1+k-1)$. If we interpret a mutation event as `success' and a coalescent event as `failure', then $M_k(t_1)$ is a geometric random variable supported on $\mathbb N_0\coloneqq\{0, 1, \ldots\}$ with success probability $(k-1)/(t_1+k-1)$ and mean $t_1/(k-1)$. It is now immediate that $S(n)$ has mean $t_1 H_{n-1}$ and variance $t_1^2H_{n-1}^{(2)}+t_1 H_{n-1}.$

			
			Alternatively, the conditional distribution of $M_k^\dagger(t_1)$ given $\tau_k$ is Poisson with parameter $t_1 k\tau_{k}/2$ for which we write
			\begin{align}
			(M_{k}^\dagger(t_1)|\tau_{k}) \sim \Pn(t_1 k\tau_{k}/2).
			\end{align}

Watterson~\cite{Watterson1975} showed a law of large numbers,
\begin{align}
  \frac{S(n)}{\log n}\to_{a.s.} t_1\qquad \text{as } n\to\infty,
\end{align}
and a central limit theorem
\begin{align}\label{eq:segregating_sites_clt}
  \frac{S(n)-\Ex S(n)}{\sqrt{\Var(S(n))}}\to_d N\qquad \text{as }n\to\infty,
\end{align}
where $N$ denotes a standard Gaussian random variable.

Our next result, Theorem~\ref{thm:functialclt_segregating_sites}, also relies on the coupling provided in section~\ref{sec:coupling}. It generalises and unifies the asymptotic results on $S(n)$ of Watterson cited above.
\begin{theorem}[Functional CLT for the number of segregating sites]\label{thm:functialclt_segregating_sites}
  As $n\to\infty$ we have weak convergence of processes
  \begin{align}\label{eq:brownian_sheet_fdd}
    \left\{\frac{S(n, t)-t_1t_2\log n}{\sqrt{\log n}}, t\in\mathbb [0, 1]^2\right\} \to \{\mathscr B(t), t \in\mathbb [0, 1]^2\}
  \end{align}
  on the space $D_2$ in the S-topology, where $\mathscr B$ is a one-dimensional Brownian sheet.
\end{theorem}

\subsection{Coupling and joint convergence}\label{sec:coupling}
Our first main result, Theorem~\ref{thm:coupling}, entails two couplings of $K(n, t)$ and $S(n, t)$ for all parameters $n\geq 2$, $t\in\mathbb R_+^2$.


Let $\mathfrak P$ denote a Poisson point process on $\mathbb R_+^2$ with intensity $\lambda\otimes \frac{1}{2}\lambda$, where $\lambda$ denotes Lebesgue measure on $\mathbb R,$ such that $\mathfrak P$ is independent of Kingman's coalescent $\Pi$. Let $L_n$ denote the total branch length of the $n$-coalescent tree. Now, for any $t\in \mathbb R_+^2$ and $k\geq 2$ set
\begin{align}
  M_k(t_1) \coloneqq \card(\mathfrak P\cap [L_{k-1}, L_k)\times [0, t_1)),
\intertext{and}
  B_k(t_1)\coloneqq \mathbf 1\{M_k(t_1)\geq 1\}.
\end{align}
This definition entails two different couplings depending on whether the sequence $(L_n, n\geq 2)$ is considered in the natural coupling or in the temporal coupling.

\begin{theorem}
The number of mutations appearing while the coalescent has $k$ lineages is equal in distribution to $M_k(t_1)$. The r.v.~$M_k(t_1)$ is geometric with success parameter $(k-1)/(t_1+k-1)$. The r.v. $B_k(t_1)$ follows a Bernoulli law such that $B_k(t_1)$ has success parameter $t_1/(t_1+k-1).$ Moreover,
\begin{align}\label{eq:definitions_k_s}
  K(n, t) &=_d 1+\sum_{k=2}^{\lfloor n^{t_2}\rfloor}B_k(t_1), \\
  S(n, t) &=_d \card(\mathfrak P\cap [0, L_k)\times [0, t_1))= \sum_{k=2}^{\lfloor n^{t_2}\rfloor} M_k(t_1).
\end{align}
\end{theorem}
Henceforth we take~\eqref{eq:definitions_k_s} as the definition of $K(n, t)$ and $S(n, t)$. Notice that in both couplings $K(n, t)\leq 1+S(n,t )$ almost surely.

We turn to our second main result.
\begin{theorem}
  As $n\to\infty$ we have weak convergence of processes
\begin{align}
  \left\{\frac{(K(n, s), S(n, t))-(s_1s_2, t_1t_2)\log n}{\sqrt{\log n}}, s, t\in [0, 1]^2\right\}\to\{(\mathscr B(s), \mathscr B(t)), s, t\in [0, 1]^2\},
\end{align}
  on the space $D_4$ in the S-topology, where $\mathscr B$ is a one-dimensional Brownian sheet on the unit square.
\end{theorem}

\section{Preliminaries}\label{sec:preliminaries}
  In this section we review some notions and facts from multiparameter processes and their convergence criteria.


\subsection{Multiparameter processes and convergence criteria}\label{sec:multiparameter_processes}
We borrow some notation and terminology from~\cite{BickelWichura1971} on multiparameter processes.

For $q\in\mathbb N$ let $T_1, \ldots, T_q$ denote subsets in $[0, 1]$, each of which contains $\{0, 1\}$ and is either finite or $[0, 1].$ Set $T\coloneqq T_1\times T_2\times\cdots\times T_q$, and call $\bigcup_{p=1}^q T_1\times \cdots\times T_{p-1}\times\{0\}\times T_{p+1}\times\cdots\times T_q$ the lower boundary of $T.$ Let $X=\{X(t)\colon t\in T\}$ be a stochastic process whose state space is some linear space $E$ (in our case this will be $\mathbb R$) endowed with a norm $\norm\cdot$. The sample paths of $X$ are assumed to be smooth enough to permit each of the supremal quantities defined below to be computed by running the time indices involved through countable dense subsets. For any $p$ and any $t\in T_p$ define $X(t)^{(p)}\colon T_1\cdots T_{p-1}\times T_{p+1}\times\cdots\times T_q\to E$ by
\begin{align*}
  X(t)^{(p)}(t_1, \ldots, t_{p-1}, t_{p+1}, \ldots, t_q)\coloneqq X(t_1, \ldots, t_{p-1}, t, t_{p+1}, \ldots, t_q). 
\end{align*}

A block $B$ in $T$ is a subset of the form $(s, t]=\bigtimes_p (s_p, t_p]$ with $s=(s_1, \ldots, s_q),$ $t=(t_1, \ldots, t_q)\in T.$ The set $\bigtimes_{\rho\neq p} (s_\rho, t_\rho]$ is called the $p$th face of $B=(s, t].$ Disjoint blocks $B, C$ are $p$-neighbours if they abut and have the same $p$th face. For any block $B=(s, t]$ let
\begin{align}
  X(B)\coloneqq \sum_{\epsilon_1\in\{0, 1\}}\cdots \sum_{\epsilon_q\in\{0, 1\}} (-1)^{q-\sum_p\epsilon_p} X(s_1+\epsilon_1(t_1-s_1), \ldots, s_q+\epsilon_q(t_q-s_q))
\end{align}
be the increment of a stochastic process $\{X(t), t\in T\}$ around $B$.

Let $\beta>1, \gamma>0$, and let $\mu$ be a finite measure on $T.$ We assume that $\mu$ assigns measure zero to the lower boundary of $T.$ Say that $(X, \mu)$ satisfies condition $C(\beta, \gamma),$ if
\begin{align}
  \Prob{m(B, C)\geq\lambda}\leq\lambda^{-\gamma}(\mu(B\cup C))^\beta
\end{align}
for all $\lambda>0$ and every pair of neighbouring blocks $B, C$ in $T,$  where $m(B, C)\coloneqq \min(\abs{X(B)}, \abs{X(C)})$. Condition $C(\beta, \gamma)$ is implied by the frequently used moment condition
\begin{align}
  \Ex[\abs{X(B)}^{\gamma_1}\abs{X(C)}^{\gamma_2}]\leq (\mu(B))^{\beta_1}(\mu(C))^{\beta_2},
\end{align}
where the gammas and betas satisfy $\gamma_1+\gamma_2=\gamma,$ $\beta_1+\beta_2=\beta$.

We turn to criteria for weak convergence of multiparameter processes as can be found in~\cite{BickelWichura1971}. 

A function $x\colon T\to\mathbb R$ is called a step function if $x$ is a linear combination of functions of the form $t\mapsto \mathbf 1_{E_1\times\cdots\times E_q}(t),$ where each $E_p$ is either a left-closed, right-open subinterval of $[0, 1]$, or the singleton $\{1\}$ and where $\mathbf 1_E$ denotes the indicator of the set $E$. Then $D_q$ is defined to be the uniform closure of the vector subspace of simple functions in the space of all bounded functions from $T$ to $\mathbb R$. We work with a metric topology on $D_q$ (S-topology) which coincides with Skorokhod's topology $J_1$ for $q=1$. Let $\Lambda$ denote the group of transformations $\lambda\colon T\to T$ of the form $\lambda(t_1, \ldots, t_q)=(\lambda_1(t_1), \ldots, \lambda_q(t_q))$, where each $\lambda_p\colon [0, 1]\to [0, 1]$ is continuous, strictly increasing, and fixes zero and one. The Skorokhod distance between $x, y\in D_q$ is defined via
\begin{align*}
  d(x, y)\coloneqq \inf\{\min(\norm{x-y\lambda}, \norm{\lambda})\colon \lambda\in\Lambda\},
\end{align*}
where $\norm{x-y\lambda}\coloneqq\sup\{\abs{x(t)-y(\lambda(t))}\colon t\in T\}$ and $\norm\lambda\coloneqq \sup\{\abs{\lambda(t)-t}\colon t\in T\}.$ $D_q$ is separable with respect to the metric topology and complete, and the Borel sigma-algebra $\mathscr D_q$ coincides with the sigma algebra generated by the coordinate mappings, cf.~\cite{Billingsley1968, Straf1969}. A sequence $(X_n)_{n\geq 1}$ of $D_q$-valued processes is said to converge in the S-topology to a $D_q$-valued process $X,$ written $X_n\to X,$ if $\Ex f(X_n)\to \Ex f(X)$ for all S-continuous bounded functions $f\colon D_q\to \mathbb R$.

We will now present some criteria of weak convergence of $D_q$-valued processes from Bickel and Wichura~\cite{BickelWichura1971}. For any finite $S\subseteq T$ define $\pi_S\colon D_q\to \mathbb R^S$ by $\pi_S(x)=(x(s))_{s\in S}$. Let $\mathscr T$ denote the collection of subsets of $T$ of the form $U_1\times\cdots\times U_q$, where each $U_p$ contains zero and one and has countable complement. For each $D_q$-valued process $X$ set
\[T_X\coloneqq \{t\in T\colon \pi_{\{t\}} \text{ is continuous a.s.~with respect to the law of $X$ on $(D_q, \mathscr D_q)$} \}.\]
\cite{Billingsley1968, Straf1969} show $T_X\in \mathscr T$.
 A partition of $T$ is called a $\delta$-grid if each element of the partition is a ``left-closed, right-open'' rectangle of diameter at least $\delta>0$, and define $w_\delta'\colon D_q\to\mathbb R$ by
\[w_\delta'(x) \coloneqq \inf_\Delta \max_{G\in\Delta} \sup_{s, t\in G} \abs{x(t)-x(s)},\]
where the infimum is taken over all $\delta$-grids $\Delta$ in $T$.

Let $(X_n)_{n\geq 1}$, and $X$ be $D_q$-valued processes, and suppose that $X$ is continuous at the upper boundary of $T.$
\begin{theorem}[\cite{Straf1972, BickelWichura1971}]
  Let $(X_n)_{n\geq 1}$ be $D_q$-valued processes. In order that the sequence $(X_n)$ converges weakly, it is necessary and sufficient that
  \begin{itemize}
    \item $(\pi_S(X_n))$ converges weakly, for all finite subsets $S$ of some member $\tau$ of $\mathcal T$, and
    \item $plim_\delta\lim_n w_\delta'(X_n)=0$
  \end{itemize}
  and then $X_n\to_d X$, where the distribution of the $D_q$-valued process $X$ is determined by $\pi_S(X_n)\to\pi_S(X)$ for all finite $S\in\tau\cap T_X.$ Here the second condition means $\lim_{\delta\downarrow 0}\limsup_n \Prob{w_\delta'(X_n)\geq \varepsilon}=0$ for all $\varepsilon>0$.
\end{theorem}
Now define $w_\delta''\colon D_q\to\mathbb R$ by
\[w_\delta''(x)\coloneqq \max_p w_\delta''^{(p)}(x),\]
where
\[w_\delta''(x)\coloneqq \sup\{\min(\norm{x(t)^{(p)}-x(s)^{(p)}}, \norm{x(u)^{(p)}-x(t)^{(p)}})\colon s\leq t\leq u, u-s\leq \delta\},\]
for $p\in[q].$ Here for each $p\in[q]$ and $t\in T_p$ $x(t)^{(p)}$ is defined by
\[x(t)^{(p)}(t_1, \ldots, t_{p-1}, t_{p+1}, \ldots, t_q)\coloneqq x(t_1, \ldots, t_{p-1}, t_p, t_{p+1}, \ldots, t_q).\]
A corollary from this result is as follows.

\begin{corollary}
  Let $(X_n)_{n\geq 1}$ and $X$ be $D_q$-valued processes, and suppose that $X$ is continuous at the upper boundary of $T$. Then in order that $X_n\to X$, it is necessary and sufficient that
  \begin{align}\label{eq:tightness_condition}
    \pi_S(X_n)\to\pi_S(X) \text{ for all finite subsets $S$ of some member $\tau$ of $\mathcal T$},\\\notag
    plim_\delta\lim_n w_\delta''(X_n) = 0.
  \end{align}
\end{corollary}

Here is another criterion that is more useful for calculations.

\begin{theorem}[Theorem 3 in~\cite{BickelWichura1971}]\label{thm:bickel_wichura3}
  Suppose that each $X_n$ vanishes along the lower boundary of $T$, and that there exist constants $\beta>1,$ $\gamma>0$ and a finite nonnegative measure $\mu$ on $T$ with continuous marginals such that $(X_n, \mu)$ satisfies condition $C(\beta, \gamma)$ for each $n$. Then the tightness condition~\eqref{eq:tightness_condition} holds.
\end{theorem}

We refer to the above convergence criteria as the Bickel-Wichura conditions.

\subsection{Brownian sheet}
Recall that a Gaussian process on a non-empty index set $T$ is a family $X\coloneqq \{X(t)\colon t\in T\}$ of Gaussian random variables such that for any sequence $t_1, \ldots, t_d\in T,$ $d\in\mathbb N,$ and any $a\in\mathbb R^d,$ $\sum_{k=1}^d a_kX(t_k)$ follows a Gaussian law. We say that $X$ is centred if all its marginals have zero mean. A one-dimensional Brownian sheet on $\mathbb R_+^2\coloneqq [0, \infty)^2$ is a two-parameter, centred Gaussian process $\mathscr B=\{\mathscr B(t), t\in\mathbb R_+^2\}$ with covariance function
\begin{align}
  \Cov(\mathscr B(s), \mathscr B(t)) &= \min(s_1, t_1)\min(s_2, t_2)\qquad\text{ for all }s, t\in\mathbb R_+^2.
\end{align}

\section{proofs}\label{sec:proofs}
We turn to the proofs of our results.

\begin{proof}(Proof of theorem~\ref{thm:coupling})
Recall that in the temporal coupling $L_k^\dagger-L_{k-1}^\dagger=k\tau_k$ is the total length of branches created when the coalescent has exactly $k$ lineages, thus $(M_k^\dagger(t_1)|\tau_k)$ obeys a Poisson law with parameter $k\tau_k t_1/2,$ in particular, $(M_k^\dagger(t_1))_{2\leq k\leq n}$ is a sequence of independent random variables. It is well known that Poisson mixture distributions with mixing parameter following an exponential law are geometric, cf.~\cite[page 212]{JohnsonKempKotz2005}. Thus $M_k^\dagger(t_1)$ obeys a geometric law with success parameter $(k-1)/(t_1+k-1)$. Setting $B_k^\dagger(t_1)\coloneqq \mathds{1}\{ M_k^\dagger(t_1)\geq 1  \},$ $k\geq 2,$ we see that
\begin{align}
  K(n, (t_1, 1)) \coloneqq 1+B_2^\dagger(t_1)+\cdots+B_n^\dagger(t_1)
\end{align}
is equal in distribution to the number of cycles in a $(n, t_1)$ Ewens permutation given by Feller's coupling.
\end{proof}

\subsection{Total branch lengths}
We study the asymptotic behaviour of $S(n, t).$ To this end, we rely on its representation via the Poisson point process $\mathfrak P$. In particular, we will condition on the total branch lengths $L_n$. Understanding the asymptotic behaviour of $L_n$ both in the natural coupling and in the temporal coupling is crucial for our proof of Theorem~\ref{thm:functialclt_segregating_sites}.

It is well-known that in the temporal coupling one has convergence a.s.~and in $L^p$ for any $p>0,$
\begin{align}
  L_n^\dagger -2\log n \to 2\mathscr L^\dagger,
\end{align}
as $n\to\infty,$ where the r.v.~$\mathscr L^\dagger$ obeys a Gumbel law, cf.~\cite{MoehlePitters2015}. Moreover, in the temporal coupling, cf.~\eqref{eq:branch_lengths_temporal_coupling}, we see that $L_n^\dagger-2H_{n-1}$ is a mean-square bounded martingale, thus $L_n^\dagger=2\log n+O_p(1).$

In the natural coupling, Pfaffelhuber, Wakolbinger and Weisshaupt~\cite{PfaffelhuberWakolbingerWeisshaupt2011} showed convergence in $L^2$
\begin{align}
  L_n^\star-2\log n\to 2\mathscr L^\star,
\end{align}
as $n\to\infty,$ where $\mathscr L^\star$ obeys a Gumbel law. We show that this convergence also holds almost surely, and that the limits $\mathscr L^\dagger$ and $\mathscr L^\star$ agree a.s.

Let us recall the hypergeometric distribution. Let $X$ denote the number of white balls in a sample of size $n\leq b+w$ drawn without replacement from a set of $b\in\mathbb N_0$ black and $w\in\mathbb N_0$ white balls. The probability mass function of $X$ is
\[ \Prob{X=k} = \frac{\binom{w}{k}\binom{b}{n-k}}{\binom{b+w}{n}}   \text{ for } k\in\mathbb \{0, \ldots, n\wedge w\}. \]
We say that $X$ obeys a hypergeometric distribution with parameters $(b, w, n)$. Moreover, the r.v.~$X$ has mean $\Ex X=nw/(b+w)$ and variance $\Var(X)=nwb(n-1)/((b+w)^2(b+w-1))$. 

Recall that $\mathbb B_k(n)$ counts the number of branch segments in Kingman's coalescent that support leaves with labels in $1, \ldots, n$ during the time that the coalescent has $k$ branches. Clearly, $\mathbb B_k(n)$ is non-decreasing in both parameters $n$ and $k$. It is known that $\mathbb B_k(n)$ has a hypergeometric distribution with parameters $(n-1, k, n)$, cf.~\cite[Lemma 4.8]{EtheridgePfaffelhuberWakolbinger2006}. 

\begin{lemma}\label{lem:convergence_of_branches}
	Fix a natural number $k\geq 2$. As $n\to\infty,$ we have almost sure convergence
	\begin{align}\label{eq:convergence_of_branches}
	  \mathbb B_k(n)\to k.
	\end{align}
\end{lemma}
\begin{proof}
	First, notice that $\Ex \mathbb B_k(n)=nk/(n+k-1)\to k$ as $n\to\infty.$
Since $\mathbb B_k(n)$ is non-negative, the convergence in~\eqref{eq:convergence_of_branches} holds in $L^1$, and the claim follows.
\end{proof}
In our following proof we will approximate $L_n^\dagger$ and $L_n^\star$ by the following random variables. For $n, m\geq 2$ set
\begin{align}
  L_{n, m} & \coloneqq \sum_{k=2}^m \mathbb B_k(n)\tau_k.
\end{align}
Notice that $L_{n, m}$ is increasing in both parameters $n$ and $m.$
\begin{proposition}
	For fixed $n\geq 2$ we have almost sure convergence $L_{n, m}\uparrow L_n^\star$ as $m\to\infty.$ For fixed $m\geq 2$ we have almost sure convergence $L_{n, m}\uparrow L_m^\dagger$ as $n\to\infty$.
\end{proposition}
\begin{proof}
	For the first claim, notice that $L_{n, m}\to L_n^\star$ in $L^1$ as $m\to\infty$ by a simple calculation. Hence this convergence holds in probability. It also holds almost surely, since $L_{n, m}$ is non-decreasing in $m$.
	The second claim is an immediate application of Lemma~\ref{lem:convergence_of_branches}.
\end{proof}


\begin{proof}(of Theorem~\ref{thm:branch_lengths_convergence})
	It suffices to show almost sure convergence $L_n^\dagger-L_n^\star\to 0$ as $n\to\infty$. Let $(\Omega, \mathcal F, \mathbb P)$ denote the probability space underlying the Kingman coalescent. We define the following sets.
	\begin{enumerate}[(i)]
		\item The sequence $L_n^\dagger-2\log n$ converges almost surely, thus
		  \[ \Omega_1\coloneqq \{ L_n^\dagger \text{ is a Cauchy sequence}  \} \]
		  has $\mathbb P$-measure $1$. Put differently, for any $\omega\in\Omega_1$ there exists $N_1\equiv N_1(\omega)\in\mathbb N$ such that $m, n\geq N_1$ implies
		  \begin{align}
		  \abs{L_n^\dagger(\omega)-L_m^\dagger(\omega)}\leq \varepsilon.
		  \end{align}
		\item Since for any $n\geq 2$ we have $L_{n, m}\uparrow L_n^\star$ almost surely as $m\to\infty$,
			\[\Omega_2\coloneqq \bigcap_{n\geq 2}\{ L_{n, m}\uparrow L_n^\star \text{ as }m\to\infty \} \]
			has $\mathbb P$-measure $1$. For any $\omega\in\Omega_2$ and any $n\geq 2$ there exists $N_2\equiv N_2(n, \omega)\in\mathbb N$ such that $m\geq N_2$ implies
			\begin{align}
			\abs{L_n^\star(\omega)-L_{n, m}(\omega)}\leq \varepsilon.
			\end{align}
			
		\item Since for any $m\geq 2$ we have $L_{p, m}\uparrow L_m^\dagger$ almost surely as $p\to\infty,$
			\[\Omega_3 \coloneqq \bigcap_{m\geq 2}\{ L_{p, m}\uparrow L_m^\dagger  \text{ as } p\to\infty\} = \bigcap_{m\geq 2}\{ L_{p, m} \text{ is a Cauchy sequence in }p \} \]
			also has $\mathbb P$-measure $1$. For any $\omega\in\Omega_3$ and any $m\geq 2$ there exists $N_3\equiv N_3(m, \omega)\in\mathbb N$ such that $p\geq N_3$ implies
			\begin{align}
			\abs{L_{p, m}(\omega)-L_m^\dagger(\omega)}\leq \varepsilon.
			\end{align}
			
	\end{enumerate}
Clearly, $\Omega'\coloneqq \Omega_1\cap\Omega_2\cap\Omega_3$ has $\mathbb P$-measure one. Fix therefore $\omega\in\Omega'$ and $\varepsilon>0$ arbitrarily. We show that for any $n\geq N_1(\omega)$ we have
\[ \abs{L_n^\dagger(\omega)-L_n^\star(\omega)} \leq 3\varepsilon. \]
In fact, choose $m\geq N_1(\omega)\vee N_2(n, \omega)$, and $p\geq N_3(m, \omega)\vee n$. Notice that
\begin{align*}
  L_n^\star(\omega)-L_n^\dagger(\omega) \leq L_n^\star(\omega)-L_{n, m}(\omega)+L_{p, m}(\omega)-L_m^\dagger(\omega)+L_m^\dagger-L_n^\dagger(\omega),
\end{align*}
since $L_{p, m}(\omega)>L_{n, m}(\omega).$ Now,
\begin{align*}
  \abs{L_n^\star(\omega)-L_n^\dagger(\omega)} &\leq \abs{L_n^\star(\omega)-L_{n, m}(\omega)}+\abs{L_{p, m}(\omega)-L_m^\dagger(\omega)}+\abs{L_m^\dagger-L_n^\dagger(\omega)}\\
  &\leq 3\varepsilon.
\end{align*}
\end{proof}
We now have the tools ready to study the asymptotic behaviour of the number of segregating sites.

\subsection{Number of segregating sites}
We now turn to the proofs of our results concerning the asymptotic behaviour of the number of segregating sites.

\begin{proof}(of Theorem~\ref{thm:functialclt_segregating_sites})
	(i) First consider a conditional version of the statement. Let $0=l(1)<l(2)<\cdots$ be real numbers satisfying $l(n)=2\log n+O(1)$ as $n\to\infty.$ For $s, t\in [0, 1]^2$ such that $0\leq s\leq t$ define the block $B$ to be $[s_1, t_1)\times [l(\lfloor n^{s_2}\rfloor), l(\lfloor n^{t_2}\rfloor))$. The random variable
\begin{align*}
  \mathcal S(n, B)\coloneqq \frac{\card(\mathfrak P\cap B)-\lambda^2(B)}{\sqrt{\log n}},
\end{align*}
due to its Poisson nature, converges in distribution to a centred Gaussian random variable, $\mathcal S(B)$ say, as $n\to\infty.$ Also, for disjoint blocks $B_1, B_2, \ldots$ the random variables $\mathcal S(n, B_1), \mathcal S(n, B_2), \ldots$ are independent. Considering linear combinations, we see that the random variables
\begin{align*}
  \mathcal S(n, t)\coloneqq \frac{\card(\mathfrak P\cap [0, l(\lfloor n^{t_2}\rfloor)\times [0, t_1))- t_1t_2\log n }{\sqrt{\log n}}, \qquad t\in [0, 1]^2
\end{align*}
are asymptotically jointly Gaussian with covariance function
\begin{align*}
  \lim_{n\to\infty}\Cov(\mathcal S(n, s), \mathcal S(n, t)) &= \lim_{n\to\infty} \Var(\mathcal S(n, s\wedge t)) = (s\wedge t)_1(s\wedge t)_2, \qquad s, t\in [0, 1]^2. 
\end{align*}
Moreover, for disjoint blocks $B, C$ in the unit square by independence
\begin{align*}
  \Ex[\mathcal S(n, B)^2\mathcal S(n, C)^2] &= \Var(\mathcal S(n, B))\Var(\mathcal S(n, C)) = \lambda^2(B)\lambda^2(C).
\end{align*}
Notice that each $\mathcal S(n, t)$ vanishes along the lower boundary of $T\equiv [0, 1]^2.$ Applying the Bickel-Wichura criterion, we obtain that $\{\mathcal S(n, t), t\in [0, 1]^2\}$ converges in the functional limit sense to the standard Brownian sheet $\{\mathscr B(t), t\in [0, 1]^2\}$.

(ii) We now replace the numbers $l(n)$ by the total branch lengths $L_n$, $n\geq 2$, of the Kingman $n$-coalescent, which are independent of $\mathfrak P$. Set
\begin{align*}
  \mathscr S(n, t)\coloneqq \frac{\card(\mathfrak P\cap [0, L_{\lfloor n^{t_2}\rfloor}) \times [0, t_1))- t_1t_2\log n }{\sqrt{\log n}}, \qquad t\in [0, 1]^2.
\end{align*}
We now distinguish between the temporal coupling and the natural coupling. Since $L_n^\dagger-2H_{n-1}$ is a martingale (cf.~\eqref{eq:branch_lengths_temporal_coupling}) and square mean-bounded, we have $L_n^\dagger=2\log n+O_p(1)$ a.s. On the other hand, as the proof in Theorem~\ref{thm:branch_lengths_convergence} shows,
\begin{align*}
  L_n^\dagger-L_n^\star \to 0
\end{align*}
almost surely as $n\to\infty$. Thus $L_n=2\log n+O_p(1)$ holds true both in the natural coupling and in the temporal coupling

Thus, for any continuous, bounded $f\colon D_2\to\mathbb R$ we have from (i)
\begin{align*}
  \Exc{f(\mathscr S(n))}{L_n, n\geq 2}\to\Ex{f(\mathscr B)},
\end{align*}
and by dominated convergence
\begin{align*}
\Ex[f(\mathscr S(n))]\to\Ex{f(\mathscr B)}.
\end{align*}
This proves Theorem~\ref{thm:functialclt_segregating_sites}.
\end{proof}

\subsection{Joint convergence}
We show our main result, Theorem~\ref{thm:main_result}.

\begin{proof}(Proof of Theorem~\ref{thm:main_result})
	From~\eqref{eq:definitions_k_s} we have
\begin{align*}
  \sup_{n\geq 2, t\in [0, 1]^2} \abs{K(n, t)-S(n, t)} &\leq \sum_{k\geq 2} (M_k(1)-B_k(1)).
\end{align*}

The mean of the right hand side is finite,
\begin{align*}
  \Ex\sum_{k\geq 2} (M_k(1)-B_k(1)) &= \sum_{k\geq 2} \frac{1}{(k-1)k}<\infty,
\end{align*}
thus
\begin{align*}
\sup_{n\geq 2, t\in [0, 1]^2} \abs{K(n, t)-S(n, t)} &<\infty \text{ a.s.}
\end{align*}

Consequently, we may replace $K(n, t)$ by $S(n, t)$ in Theorem~\ref{thm:main_result}, and the claim follows from Theorem~\ref{thm:functialclt_segregating_sites}.
\end{proof}

%
%
%
%
%
%
%

			{\bf Acknowledgements.} The author thanks Zakhar Kabluchko for stimulating discussions and for pointing out~\cite{BickelWichura1971}. He also thanks Leif D\"{o}ring for support.
			
			\bibliographystyle{alpha}
			\bibliography{literature}
			
			\end{document}